\documentclass[11 pt]{amsart}

\usepackage{amscd}
\usepackage{amsmath}
\usepackage{amsthm}
\usepackage{amsfonts}
\usepackage{amssymb}
\usepackage{amsxtra}

\def\C{\mathbb{C}}

\def\N{\mathbb{N}}

\def\R{\mathbb{R}}

 \newtheorem{thm}{Theorem}[section]
 \newtheorem{cor}[thm]{Corollary}
 \newtheorem{lem}[thm]{Lemma}
 \newtheorem{prop}[thm]{Proposition}

\newcommand{\be}{\begin{equation}}
\newcommand{\ee}{\end{equation}}
\newcommand{\bea}{\begin{eqnarray}}

\newcommand{\eea}{\end{eqnarray}}
\newcommand{\Bea}{\begin{eqnarray*}}
\newcommand{\Eea}{\end{eqnarray*}}

\newcounter{cnt1}
\newcounter{cnt2}
\newcounter{cnt3}
\newcommand{\blr}{\begin{list}{$($\roman{cnt1}$)$}
 {\usecounter{cnt1} \setlength{\topsep}{0pt}
 \setlength{\itemsep}{0pt}}}
\newcommand{\bla}{\begin{list}{$($\alph{cnt2}$)$}
 {\usecounter{cnt2} \setlength{\topsep}{0pt}
 \setlength{\itemsep}{0pt}}}
\newcommand{\bln}{\begin{list}{$($\arabic{cnt3}$)$}
 {\usecounter{cnt3} \setlength{\topsep}{0pt}
 \setlength{\itemsep}{0pt}}}
\newcommand{\el}{\end{list}}
\title[Hermite pseudo-multipliers]{On Hermite pseudo-multipliers}
\author{Sayan Bagchi and Sundaram Thangavelu}
\date{}
\begin{document}


\address{department of mathematics,
Indian Institute of Science, Bangalore - 560 012, India}
\email{sayan@math.iisc.ernet.in} \email{veluma@math.iisc.ernet.in}




\begin{abstract}
In this article we deal with a variation of a theorem of Mauceri concerning the $ L^p $ boundedness of operators $ M $ which are known to be bounded on $ L^2.$ We obtain sufficient conditions on the kernel of the operaor $ M $ so that it satisfies weighted $ L^p $ estimates. As an application we prove $ L^p $ boundedness of Hermite pseudo-multipliers.

\end{abstract}






\maketitle

\section[Introduction]{introduction}
In order to motivate the definition of Hermite pseudo-multipliers, let us briefly recall the  definition of pseudo-differential operators.
Using Fourier transform, a differential operator $ p(x,D) = \sum_{|\alpha| \leq m} a_\alpha(x) \partial^\alpha $ can be represented as
$$ p(x,D)f(x) = (2\pi)^{-n/2} \int_{\R^n} e^{i x\cdot \xi}p(x,\xi) \hat{f}(\xi) d\xi  $$
where $ p(x,\xi)  = \sum_{|\alpha| \leq m} a_\alpha(x) \xi^\alpha $  and $ \hat{f} $ is the Fourier transform of $ f $ defined by
$$ \hat{f}(\xi) = (2\pi)^{-n/2} \int_{\R^n} f(x) e^{-i x\cdot \xi} dx.$$ When we try to find a fundamental solution of $p(x,D) $ which amounts to inverting the operator we end up with more general operators of the form
$$ a(x,D)f(x) = (2\pi)^{-n/2} \int_{\R^n} e^{i x\cdot \xi}a(x,\xi) \hat{f}(\xi) d\xi  $$ where $ a(x,\xi) $ is a  general function on the phase space $ \R^n \times \R^n  $ not necessarily a polynomial in the $ \xi $ variable. Such operators are called pseudo-differential operators and they play an important role in the theory of
partial differential operators.  The function $ a(x,\xi) $ is termed as the symbol of the pseudo-differential operator $ a(x,D).$ When the symbol $ a(x,\xi) $ is independent of the $ x $ variable, say  $ a(x,\xi) = m(\xi) $  the resulting operator
$$  m(D) = (2\pi)^{-n/2} \int_{\R^n} e^{i x\cdot \xi} m(\xi)  \hat{f}(\xi) d\xi  $$  is nothing new to a harmonic analyst which is known by the name of Fourier multiplier. While Fourier multipliers and pseudo-differential operators are associated to the Fourier transform, Hermite pseudo-multipliers are associated to Hermite expansions to which we now turn.\\

Consider  the Hermite operator $H=-\Delta+ |x|^2$ on $\R^n$ where $ \Delta $ stands for the standard Laplacian. The spectral resolution of $ H $ is given by
$$H=\sum_{k=0}^\infty (2k+n) P_k$$
where for each $ k =0,1,2,... , P_k$ is the orthogonal projection of $ L^2(\R^n)$ onto
the $k$-th eigenspace of $H$ corresponding to the eigenvalue
$(2k+n)$. Given a bounded function $m$ defined on $\N$, the set of
all natural numbers, we can define an operator $m(H)$ by the
prescription
$$m(H)=\sum_{k=0}^\infty m(2k+n) P_k.$$
By the Plancherel theorem (or Parseval's identity) for Hermite expansions, it is immediate that $ m(H) $ is bounded on $ L^2(\R^n) $ if and only if  $ m $ is a
bounded function on $ \N.$  On the other hand, the boundedness of $m$ alone is not enough for $m(H)$ to extend to
$L^p(\R^n), p\neq 2$ as a bounded operator. In \cite{T} the second
author has studied this problem and obtained a sufficient condition on $m$ so that $m(H)$
extends to $L^p(\R^n)$ as a bounded operator for all $
1<p<\infty.$ See also the works \cite{GHSTV} and \cite{HTV} for weighted norm inequalities for Hermite multipliers.\\

Suppose now that we have a bounded function $m$ defined on $\R^n\times \N$.
We can define an operator $m(x,H)$ formally by setting
$$m(x,H)=\sum_{k=0}^\infty m(x, 2k+n)P_k.$$
Note that $m(x,H)$ is densely defined. It would be interesting to
see if $m(x,H)$ can be extended to $L^p(\R^n)$ as a bounded
operator. Even when $p=2$ it is not clear if $m(x,H)$ will
automatically be bounded on $L^2(\R^n)$, which is the case when $m$
is independent of the $x$-variable. In analogy with
the pseudo-differential operators $ a(x, D),$ we call such operators $ m(x,H) $ Hermite pseudo-multipliers.
The boundedness of Fourier multipliers and pseudo-differential
operators have been well studied in the literature. The celebrated
theorem of H\"{o}rmander-Mihlin deals with Fourier multipliers
whereas the theorem of Calderon-Vaillancourt deals with
pseudo-differential operators, see Theorem 2.80 in Folland \cite{F} and the references there.\\

The Hermite pseudo-multipliers also occur as Weyl transforms of radial functions on $ \C^n.$ Indeed, let $ \pi(w), w\in \C^n $ stand for the
projective representation of $ \C^n $ related to the Schrodinger representation $ \pi_1 $ of the Heisenberg group. To be more precise, the action of $ \pi(w) $ on   a function $ \varphi \in L^2(\R^n) $ is given  by
$$ \pi(u+iv) \varphi(\xi) = e^{i(u\cdot \xi +\frac{1}{2} u\cdot v)}\varphi(\xi+v).$$ It is well known that $ \pi(w) $ is a unitary operator on $ L^2(\R^n).$ If $a(x,w) $ is a function on $ \R^n \times \C^n $ which is radial in the second variable then the operator
$$  Tf(x) =  \int_{\C^n} a(x,w)\pi(w)f(x) dw $$ is a Hermite pseudo-multiplier. To see this, let us expand the function $ a(x, w) $  in terms of the Laguerre functions
$$   \varphi_k(w) = L_k^{n-1}(\frac{1}{2}|w|^2) e^{-\frac{1}{4}|w|^2}. $$  By letting  $ m(x,k) $  to stand for  the Laguerre coefficients defined by
$$ m(x,k) = \frac{k!(n-1)!}{(k+n-1)!} \int_{\C^n} a(x,w) \varphi_k(w) dw $$
we get the expansion
$$  a(x,w) = \sum_{k=0}^\infty m(x,k) \varphi_k(w).$$ The Hermite projections $ P_k $ and the Laguerre functions $ \varphi_k $ are intimately connected via the following well known formula  (see  Section 1.3 in \cite{T})
$$ \int_{\C^n} \varphi_k(w) \pi(w) dw = (2\pi)^{-n} P_k .$$  Thus we see that
$ T = \sum_{k=0}^\infty m(x,k) P_k$ is  a Hermite pseudo-multiplier.\\

Though there are quite a few papers dealing with Hermite
multipliers, there is only one paper as  far as we are aware of,
dealing with Hermite pseudo-multipliers. In \cite{E} Epperson
studided the $L^p$-boundedness of $m(x, H)$ in one dimension and proved the following result.
In order to state his result, let us set up some notation.
We define $\Delta m(x, k)=m(x, k+1)-m(x, k)$ and for $j\geq 2,
\Delta^j m(x,k)=\Delta(\Delta^{j-1}m(x, k)).$

\begin{thm} (Epperson)
Assume that the Hermite pseudo-multiplier $m(x, H)$ is bounded on
$L^2(\R)$. Suppose $\sup_{x\in\R^n}|\Delta^j m(x, k)|\leq
C_j (2k+1)^{-j}$, for $j=0, 1, 2,\ldots, 5.$ Then $ m(x,H) $ is of weak type (1,1) and consequently, bounded on $ L^p(\R) $ for $ 1 < p < 2.$
\end{thm}

Note that  in the above theorem  $m(x, H)$ is already assumed to be bounded on $L^2(\R)$.
Actually, the problem of finding a satisfactory condition on $m(x,
k)$ under which $m(x, H)$ will be bounded on $L^2(\R^n)$ is still
open (see Corollary 1.4 below for a simple minded condition).
Under some decay assumptions on the finite differences of $m(x,
k)$ (in the second variable) Epperson  managed to show that
$m(x, H)$ is bounded on $L^p(\R)$, $1< p< 2$. One of our
main results in this paper is the following result which is an analogue of Epperson's theorem when $ n \geq 2.$

\begin{thm}
Let $ n \geq 2.$ Assume that the Hermite pseudo-multiplier $m(x, H)$ is bounded on
$L^2(\R^n)$. Suppose $$\sup_{x\in\R^n}|\Delta^j m(x, k)|\leq
C_j (2k+n)^{-j},$$ for $j=0, 1, 2,\ldots,  n+1.$
Then $ m(x,H) $ is of weak type (1,1) and consequently bounded on $ L^p(\R^n)$ for $ 1 < p < 2.$

\end{thm}

Note that in the above theorem we have boundedness only for $ 1 < p \leq 2.$ Since the adjoint of a Hermite pseudo-multiplier is not necessarily a
pseudo-multiplier we cannot use duality to treat the case $ p > 2.$ In fact, if
$$ m(x,H) = \int_{\C^n} a(x,|w|) \pi(w) dw $$ then an easy calculation shows that
$$ m(x,H)^* =  \int_{\C^n} \overline{a(x-v, |u+iv|)} \pi(u+iv) dudv $$
and hence it is clear that $ m(x,H)^* $ is not a pseudo-multiplier unless $ a $ is independent of $ x $ which means $ m(x,H) $ is a Hermite multilpier.

In order to prove the boundedness of $ m(x,H) $ on $ L^p(\R^n) $ for $ p > 2 $ we need to assume some extra conditions on $ m(x,k).$
Let $A_p=A_p(\R^n)$
stand for Muckenhoupt's $A_p$-class of weight functions. In what follows we always assume that $ n \geq 2.$

\begin{thm}
Assume that the Hermite pseudo-multiplier $m(x, H)$ is bounded on
$L^2(\R^n)$. Suppose $\sup_{x\in\R^n}|\Delta^j m(x, k)|\leq
C_j (2k+n)^{-j}$, for $j=0, 1, 2,\ldots,  [\frac{n}{2}]+1$ and
assume that the partial derivatives $\frac{\partial}{\partial
x_i}m(x, k)$ also satisfy similar estimates for $ j = 0,1,2, \cdots [\frac{n}{2}].$ Then for any $2< p<
\infty$ and $w\in A_{p/2}$ we have the weighted norm inequality
$$\int_{\R^n} |m(x, H)f(x)|^p w(x) dx\leq C~ \int_{\R^n} |f(x)|^pw(x)dx$$
for all $f\in L^p(\R^n, wdx)$.

\end{thm}
Note that as in the paper of Epperson we have started with a
pseudo-multiplier which is already bounded on $L^2(\R^n)$.
However, the number of finite differences involved is almost
optimal. Also note that the weight function $w$ is taken from
$A_{p/2}$, not from $A_p$ as one would expect. If we increase the
number of finite differences to $n+1$ instead of $[\frac{n}{2}]
+1$ then we can allow $A_p$ weights in the weighted norm
inequality.\\

\begin{thm}
Assume that the Hermite pseudo-multiplier $m(x, H)$ is bounded on
$L^2(\R^n)$. Suppose $\sup_{x\in\R^n}|\Delta^j m(x, k)|\leq
C_j (2k+n)^{-j}$, for $j=0, 1, 2,\ldots,  n+1$ and
assume that the partial derivatives $\frac{\partial}{\partial
x_i}m(x, k)$ also satisfy similar estimates for $ j = 0,1,2,\cdots n.$ Then for any $ 1< p<
\infty$ and  $w\in A_{p}$ we have
$$\int_{\R^n} |m(x, H)f(x)|^p w(x) dx\leq C~ \int_{\R^n} |f(x)|^pw(x)dx$$
for all $f\in L^p(\R^n, wdx)$.

\end{thm}

It would be interesting to find a sufficient condition on the
multiplier $m(x, k)$ so that $m(x, H)$ is bounded on $L^2(\R^n)$.
This would amount to proving an analogue of
Calderon-Vaillancourt's theorem for Hermite expansions. Recently
Ruzhansky \cite{RV} and his collaborators have looked at pseudo
differential operators on compact Lie groups and proved certain $ L^2 $ boundedness results. As we deal with a noncompact situation, their proof cannot be adapted to treat Hermite pseudo-multipliers.\\

However, we do have some examples of Hermite pseudo-multipliers which are bounded on $  L^2(\R^n).$ Let $ a(x,w) $ be a function on $ \R^n \times \C^n $ which is radial in the second variable. If we further assume that
$$ \int_{\C^n}  \sup_{x\in\R^n}|a(x,w)| dw < \infty, $$ then it is not difficult to show that the Hermite pseudo-multiplier
$$ Tf(x) = \int_{\C^n}  a(x,w) \pi(w)f(x) dw $$ is bounded on $ L^2(\R^n).$ Another interesting class of $ L^2 $ bounded pseudo-multipliers is given by the following consideration. Suppose $a(x,t)$ is a bounded function on $\R^n\times \R$ which is
$2\pi$-periodic in $t$. If we define
$$m(x,k)=\int_0^{2\pi} a(x,t) e^{-ikt}dt=\hat{a}(x,k),$$
then it can be shown that $m(x,H)$ is bounded on $L^2(\R^n)$.
Moreover, it is possible to translate the conditions on $m$ in
Theorem 1.1 into conditions on the function $a(x,t)$. Thus we have
the following
\begin{cor}
Let $N\geq [\frac{n}{2}]+1$ be an integer. Suppose $a(x,t)$ is a
function on $\R^n\times \R$ which is once differentiable in the
$x$-variable and $N$ times differentiable in the t-variable.
Assume that for $j= 0, 1, 2, \ldots, N$ and $i=0, 1, 2, \ldots, n$
both $ \frac{\partial^j}{\partial t^j} a(x, t)$ and $ \frac{\partial^j}{\partial t^j}
\frac{\partial}{\partial x_i} a(x,t)$ are bounded functions on
$\R^n\times \R$. Then for every $ w\in A_{p/2} $ and $ p > 2, \hat{a}(x, H)$ satisfies the weighted norm
inequality
$$\int_{\R^n}|\hat{a}(x,H)f(x)|^p w(x)dx\leq C~ \int_{\R^n} |f(x)|^p w(x)dx$$
for all $f\in L^p(\R^n, wdx)$.
\end{cor}

 We remark that the condition on $a$ can be replaced by the weaker
 assumption
 $$\sup_{x,t}(|D^ja(x,t)|+|D^j \nabla_x a(x, t)|)\leq C_j, ~~~~ 0 \leq j \leq N, $$
 where $D a(x, t)= \partial_t(e(t)a(x,t))$ and $  e(t)=(e^{-it}-1).
 $ This is immediate from the fact that
 $$ik \Delta \hat{a}(x,k)=\int_0^{2\pi} Da(x,t)e^{-ikt}dt.$$
 It would be interesting to see if the assumption on $\nabla_x
 a(x, t)$ can be dispensed with. Right now we do not know how to
 do that as our method of proving Theorem 1.1 requires the
 assumption on $\nabla_x a$.\\

 The problem of proving $L^p$ boundedness of an operator $M$ which
 is known to be bounded on $L^2(\R^n)$ has been studied by
 Mauceri \cite{M}. Using a transference technique and a theorem on Weyl
 multipliers he found a sufficient condition on $M$ so that it
 extends to a bounded operator on $L^p(\R^n)$. Since we use many
 ideas from his paper it is worthwhile to recall his result. The
 hypothesis on $M$ involves non-commutative derivatives $\delta_j$
 and $\bar{\delta}_j$ defined by
 $$\delta_j M=[M, A_j], \; \bar{\delta}_j= -[M, A^*_j]$$
 where $A_j= \frac{\partial}{\partial \xi_j}+\xi_j$ and $A_j^*= -\frac{\partial}{\partial \xi_j}+ \xi_j, \; j=1, 2, \ldots, n.$
 Higher order (non-commutative) derivatives $\delta^\alpha$ and
 $\bar{\delta}^\beta, \alpha, \beta\in \N^n,$ are defined in the usual
 way. Let $\chi_N$ stand for the dyadic projection $\sum_{2^{N-1}\leq
 2k+n< 2^N} P_k$. Mauceri's condition on $M$ reads as
 $$\sup_{N\in \N} 2^{N(|\alpha|+|\beta|-n)}||(\delta^\alpha \bar{\delta}^\beta M)\chi_N||^2_{HS}\leq C$$ where $ \|\cdot \|_{HS} $ stands for the Hilbert-Schmidt operator norm. If the above condition is satisfied for all $\alpha, \beta$ with
 $|\alpha|+|\beta|\leq n+1$, then he has proved that $M$ extends to
 $L^p(\R^n)$ as a bounded operator for $1 < p \leq 2$.\\

 Note that the number of derivatives involved is $n+1$ since the
 result is deduced from a corresponding theorem on Weyl multipliers
 on $\C^n\cong \R^{2n}$. We have every reason to believe that
 $[\frac{n}{2}] +1 $ derivatives suffice for the $L^p$-
 boundedness.\\

 Let us examine Mauceri's condition when $M=m(H)$ is a Hermite
 multiplier. In this case the kernel of $M\chi_N$ is given by
 $$(M\chi_N)(x,y)= \sum_{2^{N-1}\leq 2k+n< 2^N} m(2k+n) \Phi_k(x,y)$$
 where $\Phi_k(x,y)= \sum_{|\alpha|=k} \Phi_\alpha(x)
 \Phi_\alpha(y)$ is the kernel of $P_k$. Here $\Phi_\alpha$'s stand
 for the normalised Hermite functions on $\R^n$. Thus Mauceri's
 condition with $\alpha=\beta=0$ reads as
 $$\sum_{2^{N-1}\leq 2k+n< 2^N}|m(2k+n)|^2 (2k+n)^{n-1}\leq C~ 2^{Nn}$$
 where we have used the fact that
 $$\int_{\R^n\times \R^n}(\Phi_k(x,y))^2 dxdy=\frac{(k+n-1)!}{k!(n-1)!}\approx (2k+n)^{n-1}.$$
 On the other hand
 $$\int_{\R^n}|(M\chi_N)(x,y)|^2dy=\sum_{2^{N-1}\leq 2k+n< 2^N} |m(2k+n)|^2 \Phi_k(x,x).$$
 Since we have $\Phi_k(x,x)\leq C (2k+n)^{\frac{n}{2}-1}$ (see Lemma 3.2.2 in \cite{T})
 $$\int_{\R^n}|(M\chi_N)(x,y)|^2dy\leq C~ 2^{-\frac{n}{2} N}||M\chi_N||^2_{HS}$$
 and Mauceri's condition implies the better estimate
 $$\int_{\R^n}|(M\chi_N)(x,y)|^2dy\leq C~ 2^{N\frac{n}{2}}.$$
 We can also check that for any $\alpha, \beta\in \N^n$
$$ \int_{\R^n}|((\delta^\alpha \bar{\delta}^\beta
 M)\chi_N)(x,y)|^2dy\leq C~ 2^{N(\frac{n}{2}-|\alpha|-|\beta|)}$$
 under Mauceri's conditions on $\delta^\alpha \bar{\delta}^\beta
 M$.\\

We work with the above relatively weaker conditions on $M$ and
prove the following result. We say that a bounded linear operator
$T$ on $L^2(\R^n)$ is of class $C^k$ if $\delta^\alpha
\bar{\delta}^\beta T$ is bounded on $L^2(\R^n)$ for all
$|\alpha|+|\beta|\leq k$. For a given Hilbert-Schmidt operator
$T$, we denote by $T(x,y)$ its kernel, which is an element of
$L^2(\R^n\times \R^n)$. Note that $(\delta^\alpha
\bar{\delta}^\beta M)\chi_N$ is Hilbert-Schmidt whenever
$\delta^\alpha \bar{\delta}^\beta M$ is bounded on $L^2(\R^n)$ and
hence has an $L^2$-kernel.\\

\begin{thm}
Let $M\in B(L^2(\R^n))$ be of class $C^{[\frac{n}{2}]+1}$ and
satisfy the estimates
$$\sup_{x\in\R^n} \int_{\R^n}|((\delta^\alpha \bar{\delta}^\beta M)\chi_N)(x,y)|^2dy\leq C~ 2^{N(\frac{n}{2}-|\alpha|-|\beta|)}$$
for all $|\alpha|+|\beta|\leq [\frac{n}{2}]+1$ and $N\in \N$.
Assume that $M^*$ also satisfies the same estimates. Then $M$ can
be extended to $L^p(\R^n)$ as a bounded operator for $1<p<\infty$.
\end{thm}

If we increase the number of non-commutative derivatives, then we
can dispense with the condition on $M^*$. In view of the remarks
made earlier regarding Mauceri's condition, the following result
can be considered as the analogue of Mauceri's theorem wherein his
condition is replaced by kernel estimates.
\begin{thm}
Let $M\in B(L^2(\R^n))$ be of class $C^{n+1}$ and satisfy the
estimates
$$\sup_{x\in \R^n}\int_{\R^n}|((\delta^\alpha \bar{\delta}^\beta M)\chi_N)(x,y)|^2 dy\leq C~ 2^{N(\frac{n}{2}-|\alpha|-|\beta|)}$$
for all $|\alpha|+|\beta|\leq n+1$ and $ N \in \N$. Then
$M$ can be extended to $L^p(\R^n)$ as a bounded operator for $1<
p<\infty$. Moreover, $M$ is of weak type $(1,1)$.
\end{thm}

We also have weighted norm inequalities for the operator $M$.
\begin{thm}
Let $M\in B(L^2(\R^n))$ be of class $C^{[\frac{n}{2}]+1}$ and
satisfy the estimates
$$\sup_{x\in \R^n} \int_{\R^n} |((\delta^\alpha \bar{\delta}^\beta M)\chi_N)(x,y)|^2dy\leq C~2^{N(\frac{n}{2}-|\alpha|-|\beta|)}.$$
for all $|\alpha|+|\beta|\leq
[\frac{n}{2}]+1$ and $ N \in \N.$ Then for every  $w\in A_{p/2}(\R^n)$ and $p>2,~M$ satisfies the weighted norm inequality
$$\int_{\R^n} |Mf(x)|^p w(x)dx \leq C~ \int_{\R^n}|f(x)|^pw(x)dx$$
for all $f\in L^p(\R^n)$.
\end{thm}

Note that in the above theorem there are restrictions on $p$ and
the weight function. Once again, by increasing the number of
derivatives in the hypothesis on $M$ we can obtain the following
result
\begin{thm}Let $M\in B(L^2(\R^n))$ be of class $C^{n+1}$
and satisfy the estimates
$$\sup_{x\in \R^n} \int_{\R^n} |((\delta^\alpha \bar{\delta}^\beta M)\chi_N)(x,y)|^2dy\leq C~2^{N(\frac{n}{2}-|\alpha|-|\beta|)}.$$
for all $|\alpha|+|\beta|\leq n+1$ and $ N \in \N$. Then for every $w\in A_p(\R^n)$ and $1<p<\infty,
 ~M$ satisfies the weighted norm inequality
$$\int_{\R^n} |Mf(x)|^p w(x)dx \leq C ~ \int_{\R^n} |f(x)|^pw(x)dx$$
for all $f\in L^p(\R^n)$.
\end{thm}

We remark that Corollary 1.5 also has a version under the stronger hypothesis used in the above theorem.
The plan of the paper is as follows. In Section 2, we prove good
estimates on the kernel $M$ which are then used in Section 3 to
prove the weighted norm inequalities stated in Theorems 1.8 and
1.9. We prove  Theorems 1.6 and 1.7  in Section 4 and finally, in Section 5, we apply these results and techniques
used in their proofs to prove our results on Hermite
pseudo-multipliers.\\

\section[Estimating the kernels]{Estimating the kernels}
In an earlier paper \cite{BT} we studied weighted norm
inequalities for Weyl multipliers by modifying the methods used by
Mauceri in \cite{M}. Some ideas from \cite{BT} and \cite{M} will
be used in this paper. As in \cite{BT} we let $t_j=2^{-j}, j= 1,
2, \cdots $ and consider
$$S_j=\sum_{k=0}^\infty (e^{-(2k+n) t_{j+1}}- e^{-(2k+n) t_j})P_k = \varphi_j(H)$$
where $ \varphi_j(x) = (e^{-t_{j+1}x}-e^{-t_jx}).$ Then it follows that $\sum^N_{j=1}S_j=
e^{-t_{N+1}H}-e^{-t_1H}$ and taking limit as $N\rightarrow
\infty $ we get $I=\sum^\infty_{j=0} S_j $ with $ S_0 = e^{-t_1H}$. Using
this we decompose our operator $M$ as
$$M=\sum_{j=0}^\infty M_j,\;\;\; M_j=MS_j, j =0,1,2,\ldots .$$
In order to prove good estimates for the kernels of $M_j$ we
require the following proposition which is a variant of a result
found in Mauceri \cite{M}.
\begin{prop}
For any multi-indices $\gamma$ and $\rho$ we have
\be||\chi_N(\delta^\gamma \bar{\delta}^\rho S_j)||_{op}\leq C~
(2^N t_{j+1})2^{-N(|\gamma|+|\rho|)/2} f_{\gamma, \rho}(2^Nt_{j+1});\ee
\be\sup_{x\in \R^n}\int_{\R^n}|(\chi_N \delta^\gamma
\bar{\delta}^\rho S_j)(x,y)|^2 dy\leq C~
(2^N t_{j+1})^2 2^{N(\frac{n}{2}-|\gamma|-|\rho|)}f_{\gamma, \rho}(2^N t_{j+1})\ee
where $f_{\gamma, \rho}$ is a rapidly decreasing function.
\end{prop}

In order to prove the above proposition we make use of the following
Lemma proved in Mauceri \cite{M} (see Lemma 2.1). Let $ D_- $ and
$ D_+ $ be the backward and forward finite difference operators
defined by
$$ D_-\phi(k) = \phi(k)-\varphi(k-2),~~~~ D_+\phi(k) =\phi(k+2)-\phi(k).$$ Higher order finite differences $D_+^r$ and $D_-^s$ are
defined recursively. With these notations we have
\begin{lem}
Given multi-indices $\gamma$ and $\rho$ there exist constants
$C_{\gamma, \rho, \alpha}$ such that
$$\delta^\gamma \bar{\delta}^\rho \phi(H)=\sum C_{\gamma, \rho, \alpha}(A^*)^{\rho+\alpha-\gamma} A^\alpha (D^{|\alpha|}_-D_+^{|\rho|}\phi)(H)$$
where the sum is taken over all multi-indices $\alpha$
satisfying $\alpha\leq \gamma\leq \rho+\alpha$.
\end{lem}

Using this lemma we can now prove Proposition 2.1. Since
$S_j=\phi_j(H)$, $\phi_j(x)=(e^{-t_{j+1}x}- e^{-t_jx})$ we have

$$\chi_N(\delta^\gamma \bar{\delta}^\rho S_j)=
\sum_{\alpha\leq \gamma\leq
\rho+\alpha}C_{\gamma,\rho.\alpha} \chi_N (A^*)^{\rho+\alpha-\gamma}A^\gamma(D^{|\alpha|}_-
D^{|\rho|}_+ \phi_j)(H).$$ It is enough to estimate the operator
norm of $ \chi_N(A^*)^{\rho+\alpha-\gamma}A^\alpha D_-^{|\alpha|}
D_+^{|\rho|}\phi_j(H)$ for each $\alpha$, $\alpha\leq \gamma \leq
\rho+\alpha$. Since $A_j^*
\Phi_\mu=(2\mu_j+2)^{\frac{1}{2}}\Phi_{\mu+e_j}$ and
$A_j\Phi_\mu=(2\mu_j)^{\frac{1}{2}}\Phi_{\mu-e_j}$, it is clear
that the above operator is a weighted shift operator and it is
enough to estimate
$$\sup_{|\mu|\sim 2^N}||(A^*)^{\rho+\alpha-\gamma} A^{\alpha} D_-^{|\alpha|} D^{|\rho|}_+ \phi_j(H)\Phi_{\mu}||_2.$$
The above is clearly bounded by a constant times
$$ 2^{\frac{N}{2}(|\rho|+2|\alpha|-|\gamma|)}
 |D^{|\alpha|}_-D^{|\rho|}_+\phi_j(2|\mu|+n)| $$
where $|\mu|\sim 2^N$. Let $m=|\alpha|+|\rho|$. Estimating the
finite differences in terms of derivatives we have
$$|D^{|\alpha|}_-D^{|\rho|}_+ \phi_j(2|\mu|+n)|\leq C~|\phi_j^{(m)}(2|\mu|+n)|.$$
Recalling that $\phi_j(x)=e^{-t_{j+1}x}-e^{-t_jx}$, we see that
$$(-1)^m \phi_j^{(m)}(x)=t_{j+1}^m e^{-t_{j+1} x}-t_j^m e^{-t_j x}$$
which we rewrite as (since $t_j=2 t_{j+1}$)
$$(-1)^m \phi_j^{(m)}(x)=t_{j+1}^m (e^{-t_{j+1}x}-e^{-2t_{j+1}x})+t_{j+1}^m(1-2^m)e^{-t_{j}x}$$
By mean value theorem \Bea
|e^{-t_{j+1}x}-e^{-2t_{j+1}x}|&=&|e^{-t_{j+1}x}(1-e^{-t_{j+1}x})|\\
&\leq& ~ xt_{j+1} e^{-t_{j+1}x}.\Eea Since $|\mu|\sim 2^N$ we get
the estimate
$$t_{j+1}^m|e^{-t_{j+1}(2|\mu|+n)}-e^{-2t_{j+1}(2|\mu|+n)}|$$
$$\leq C~ t_{j+1} 2^N  2^{-Nm}(2^N t_{j+1})^m e^{-c t_{j+1}2^N}$$ for some $ c > 0.$
Similarly, the other term gives the estimate
$$t_{j+1}^m e^{- t_{j+1}x}\leq C~ t_{j+1} 2^N  2^{-Nm}(2^N t_{j+1})^{m-1}e^{-c 2^N t_{j+1}}.$$
Thus we have proved
$$|D^{|\alpha|}_-D^{|\rho|}_+ \phi_j(2|\mu|+n)|\leq C~ 2^N t_{j+1} 2^{-N(|\alpha|+|\rho|)}g_{\alpha,\rho}(2^N t_{j+1})$$
with $g_{\alpha, \rho}(x)=x^m e^{-cx}+x^{m-1}e^{-cx}$, for some
$c>0.$ Putting together the typical term is bounded from above by
$$C~ 2^{N/2(|\rho|+2|\alpha|-|\gamma|)}2^N t_{j+1} 2^{-N(|\alpha|+|\rho|)}g_{\alpha, \rho}(2^N t_{j+1}).$$
which is the required one. As this is true for any $\alpha\leq
\gamma\leq \rho+\alpha, $ part (1) of the proposition is proved. Note that the function $ f_{\gamma,\rho} $ appearing in the proposition is the sum of
$ g_{\alpha, \rho},$ the sum being extended over all $ \alpha $ satisfying $\alpha\leq
\gamma\leq \rho+\alpha.$ Consequently, $ f_{\gamma,\rho} $ has exponential decay a fact which will be used later.

The proof of the second part of the proposition is similar. Again,
by the lemma of Mauceri we only need to estimate the kernels of
$$\chi_N (A^*)^{\rho+\alpha-\gamma}A^\alpha D^{|\alpha|}_-D^{|\rho|}_+\phi_j(H).$$
Since the kernel of $(A^*)^{\rho+\alpha-\gamma}A^\alpha
D^{|\alpha|}_- D_+^{|\rho|}\phi_j(H)$ is
$$\sum_{\mu}(D_-^{|\alpha|}D_+^{|\rho|}\phi_j)(2|\mu|+n)(A^*)^{\rho+\alpha-\gamma}A^{\alpha} \Phi_{\mu}(x) \Phi_{\mu}(y)$$
we are required to estimate $$\sum_{|\mu|\sim
2^N}|D^{|\alpha|}_-D^{|\rho|}_+ \phi_j(2|\mu|+n)|^2
|(A^*)^{\rho+\alpha-\gamma}A^{\alpha}\Phi_{\mu}(x)|^2.$$
Proceeding as before and using the fact that (see Lemma 3.2.2 in \cite{T})
$$\sum_{|\mu|=k}\Phi_{\mu}(x)^2\leq C~ (2k+n)^{\frac{n}{2}-1},$$
the above is bounded by
$$C~ t_{j+1}^2 2^{2N} 2^{N(\frac{n}{2}-|\gamma|-|\rho|)}g_{\alpha, \rho}(2^N t_{j+1})^2.$$
This completes the proof of the proposition.\\

Using the results of the above proposition we get the following
estimate on the kernel of $M_j$.
\begin{prop}
Let $ M $ satisfy the hypothesis stated in Theorem 1.8. Then for
all $l\in \N, l\leq [\frac{n}{2}]+1$ we have
$$\sup_{x\in\R^n} \int_{\R^n}|x-y|^{2l}|M_j(x,y)|^2dy\leq C~ t_{j+1}^{l-\frac{n}{2}}$$

\end{prop}
\begin{proof}
First observe that $(x_i-y_i)M_j(x,y)$ is the kernel of the
commutator $[x_i, M_j]$ (where $x_i$ stands for the operator of
multiplication by $x_i$). Since $x_i=\frac{1}{2}(A_i+A^*_i)$ it
follows that $[x_i, M_j]$ can be expressed in terms of the
derivatives $\delta_i M_j$ and $\bar{\delta}_i M_j$. Therefore, in
order to prove the proposition , it is enough to estimate the
kernels of $\delta^\alpha \bar{\delta}^\beta M_j$ for all $\alpha,
\beta \in \N^n$ with $|\alpha|+|\beta|=l$.

By Leibnitz formula for the non-commutative derivatives
$\delta^\alpha \bar{\delta}^\beta M_j$ is a finite sum of terms of
the form $(\delta^\mu \bar{\delta}^\nu M)(\delta^\gamma
\bar{\delta}^\rho S_j)$ with $|\mu|+|\nu|+|\gamma|+|\rho|=l$ and
hence it is enough to estimate the kernel of each of the above
operators. We take one such term and split it as
$$(\delta^\mu \bar{\delta}^\nu M)(\delta^\gamma
\bar{\delta}^\rho S_j)=\sum_{N=0}^\infty (\delta^\mu
\bar{\delta}^\nu M)\chi_N. \chi_N(\delta^\gamma \bar{\delta}^\rho
S_j).$$ Thus the $L^2$ norm of the kernel (in $y$-variable) of the
left hand side is dominated by the infinite sum of $L^2$-norms of
kernels of the operators in the sum. The kernel of $(\delta^\mu
\bar{\delta}^\nu M)\chi_N . \chi_N(\delta^\gamma \bar{\delta}^\rho
S_j)$ is given by \be \int_{\R^n}((\delta^\mu\bar{\delta}^\nu
M)\chi_N)(x,y') (\chi_N(\delta^\gamma \bar{\delta}^\rho S_j))(y',
y)dy'.\ee For $x\in \R^n$ fixed, if we let
$\phi_x(y')=((\delta^\mu \bar{\delta}^\nu M)\chi_N)(x, y')$ then
we can rewrite the above as
$$\int_{\R^n} (\chi_N(\delta^\gamma \bar{\delta}^\rho
S_j))^*(y, y')\phi_x(y')dy'= (\chi_N(\delta^\gamma
\bar{\delta}^\rho S_j))^*\phi_x(y)$$ where $(\chi_N(\delta^\gamma
\bar{\delta}^\rho S_j))^* $ is the adjoint of
$\chi_N(\delta^\gamma \bar{\delta}^\rho S_j)$. Hence the $ L^2 $ norm of the integral
(2.3) in the $ y $-variable is bounded by
$$|| (\chi_N(\delta^\gamma \bar{\delta}^\rho
S_j))^*||_{op}\left(\int_{\R^n}|((\delta^\mu \bar{\delta}^\nu
M)\chi_N)(x,y)|^2dy\right)^{\frac{1}{2}}.$$ Since $ \|T^*\| = \|T\| $ for any bounded linear operator, the hypothesis on $M$
and Proposition 2.1 lead to the estimate
$$C~ t_{j+1} 2^{(2-|\gamma|-|\rho|)N/2} f_{\gamma, \rho}(2^N t_{j+1}) 2^{(\frac{n}{2}-|\mu|-|\nu|)N/2}$$
$$=C~ t_{j+1} 2^{(\frac{n}{2}+2-l)N/2} f_{\gamma, \rho}(2^N t_{j+1}).$$
Thus the $L^2$ norm of the kernel of $(\delta^\mu\bar{\delta}^\nu
M)(\delta^\gamma \bar{\delta}^\rho S_j)$ is dominated by $$C~
t_{j+1}\sum_{N=0}^\infty 2^{N(\frac{n}{2} +2-l)/2} f_{\gamma,
\rho}(2^N t_{j+1})$$
$$\leq C~ t_{j+1}^{(l-\frac{n}{2})/2}\sum_{N=0}^\infty (2^N t_{j+1})^{(\frac{n}{2}+2-l)/2} f_{\gamma, \rho}(2^N t_{j+1}).$$
Since $l\leq [\frac{n}{2}]+1$ the series converges and hence we
obtain the required estimate.
\end{proof}
\begin{cor}
$$\int_{\R^n} |x-y|^{n+1}|M_j(x,y)|^2dy\leq C~
t_{j+1}^{\frac{1}{2}}.$$

\end{cor}
\begin{proof}
When $n$ is even by taking $l=\frac{n}{2}$ and $l=\frac{n}{2}+1$
in the above proposition we get the estimates
$$\int_{\R^n} |x-y|^n |M_j(x,y)|^2dy\leq C$$
and
$$\int_{\R^n} |x-y|^{n+2}|M_j(x,y)|^2dy\leq C~ t_{j+1}.$$
Since
$$\left(\int_{\R^n} |x-y|^{n+1}|M_j(x,y)|^2dy\right)$$
$$\leq \left(\int_{\R^n} |x-y|^{n}|M_j(x,y)|^2dy\right)^{\frac{1}{2}}\left(\int_{\R^n} |x-y|^{n+2}|M_j(x,y)|^2dy\right)^{\frac{1}{2}}$$
we get the required estimate. When $n$ is odd simply take
$l=\frac{n+1}{2}$ in the proposition.
\end{proof}
\begin{prop}
For any $ K > 0 $ there exists $ C_K > 0 $ such that $$\sup_{x\in \R^n} \int_{|x-y|\leq K}|x-y|^{n+1}|\nabla_x
M_j(x,y)|^2dy\leq C_K~ t_{j+1}^{-\frac{1}{2}}.$$
\end{prop}
\begin{proof}
For any $i=1, 2, \cdots, n$, $\frac{\partial}{\partial
x_i}M_j(x,y)$ is the kernel of $\frac{\partial}{\partial x_i}M_j$
which we can write as $\frac{1}{2}(A_i+A^*_i)M_j$. It is therefore
enough to prove the estimates for the kernels of $A_i M_j$ and
$A^*_i M_j$. We will consider $A_i M_j$; the other case is
similar. Since $A_i M_j= \delta_i M_j+M_j A_i$, we consider
$\delta_i M_j$ and $M_j A_i$ separately.\\

Recall that by our assumption, $\delta_i M $ satisfies the estimates
$$\int_{\R^n} |(\delta^\alpha \bar{\delta}^\beta (\delta_i M)\chi_N)(x,y)|^2dy\leq C~ 2^{N(\frac{n}{2}-|\alpha|-|\beta|-1)}$$
for all $|\alpha|+|\beta|\leq [\frac{n}{2}]$. Since $\delta_i
M_j= \delta_i(MS_j) = (\delta_i M)S_j+M(\delta_i S_j) $ repeating the proof of
Proposition 2.3 we obtain
$$\left(\int_{\R^n} |x-y|^{2l}|\delta_i M_j(x,y)|^2dy\right)^{\frac{1}{2}}\leq C  t_{j+1}^{(l+1-\frac{n}{2})/2}$$
for all $l\leq [\frac{n}{2}]$. When $n$ is even, taking
$l=\frac{n}{2}$ we get the estimate
$$\left(\int_{\R^n} |x-y|^n |\delta_i M_j(x,y)|^2 dy\right)^{\frac{1}{2}}\leq C t_{j+1}^{\frac{1}{2}}.$$
Since $|x-y|^{n+1}\leq K |x-y|^n$ for $|x-y|\leq K $ and
$t_{j+1}\leq 1$ we get the required estimate in this case. When
$n$ is odd, we take $l=\frac{n-1}{2}$ which gives
$$\left(\int_{\R^n} |x-y|^{n-1} |\delta_i M_j(x,y)|^2 dy\right)^{\frac{1}{2}}\leq C~ t_{j+1}^{\frac{1}{4}}$$
which again leads to the required estimate.\\

In order to estimate the kernel of $M_j A_i= M(S_j A_i)$ we use
the decomposition
$$(\delta^\alpha \bar{\delta}^\beta M)(\delta^\gamma \bar{\delta}^\rho (S_jA_i))=
\sum_{N=0}^\infty (\delta^\alpha \bar{\delta}^\beta M)\chi_N.
\chi_N (\delta^\gamma \bar{\delta}^\rho(S_j A_i)).$$ Since $A_i
\Phi_{\alpha}(x) = (2\alpha_j)^{\frac{1}{2}}\Phi_{\alpha-e_i}(x)$,
$e_i=(0, \cdots, 1, \cdots, 0)$, the operator norm of
$\chi_N(\delta^\gamma \bar{\delta}^\rho(S_j A_i))$ can be
estimated as in Proposition 2.1 to obtain
$$||\chi_N(\delta^\gamma \bar{\delta}^\rho(S_j A_i))||_{op}\leq C~ t_{j+1} 2^{(3-|\gamma|-|\rho|)N/2} f_{\gamma, \rho}(2^N t_{j+1}).$$
This extra factor of $2^{\frac{N}{2}}$ appearing on the right hand
side leads to the estimate, as in Proposition 2.3,
$$\int_{\R^n} |x-y|^{2l}|M_j A_i(x,y)|^2dy\leq C~ t_{j+1}^{l-\frac{n}{2}-1}.$$
When $n$ is odd, $l=\frac{n+1}{2}$ gives the estimate
$$\int_{\R^n}|x-y|^{n+1} |M_j A_i(x,y)|^2dy\leq C~ t_{j+1}^{-\frac{1}{2}};$$
when $n$ is even we interpolate between $l=\frac{n}{2}$ and
$l=\frac{n}{2}+1$ as in Corollary 2.4 to get the required
estimate.

\section[Weighted norm estimates for $M$]{Weighted norm estimates
for $M$} In this section we prove Theorem 1.8 by establishing the
following results. Let $\Lambda^\sharp f$ stand for the sharp
maximal function and $\Lambda_2 f= (\Lambda|f|^2)^{\frac{1}{2}}$
where $\Lambda$ is the Hardy-Littlewood maximal function.
\begin{thm}
Let $M$ satisfy the hypothesis of Theorem 1.8. Then we have
$$\Lambda^\sharp((\sum_{j=0}^N M_j)f)(x)\leq
C~(\Lambda_2f(x)+\Lambda(\Lambda_2 f)(x))$$
for any $f\in L^p(\R^n)$ and $
N\in \N$ where $C$ is independent of $N$.
\end{thm}
For any  $f\in L^p(\R^n)$, $S_jf$ is a
Schwartz function and therefore, $M_jf = M(S_jf)$ is well defined as an $L^2$
function. Moreover, each $M_j$ is an integral operator and
consequently $\sum_{j=0}^N M_j$ is also such an operator. We write
$T_N$ for $\sum_{j=0}^N M_j$ and let $K_N(x,y)$ stand for its
kernel. Thus
$$T_Nf(x)=\int_{\R^n} K_N(x, y)f(y)dy, \;\;\;\;\; K_N(x,y)=\sum_{j=0}^N M_j(x,y)$$
We will prove that $\Lambda^\sharp(T_Nf)(x)\leq C( \Lambda_2f(x)+
\Lambda(\Lambda_2f)(x))$ uniformly in $N$.\\

Given a cube $Q$ containing $x\in \R^n,$ let $2Q$ stand for the
double of $Q$. Define $f_1=f \chi_{2Q}$ and $ f_2=f-f_1$. We also
let $T_N^{(1)}$ and $T_N^{(2)}$ to stand for operators with
kernels $K_N^{(1)}(y,z)=K_N(y,z)\phi(y-z)$ and
$K_N^{(2)}(y,z)=K_N(y,z)(1-\phi(y-z))$ for a fixed  $\phi\in
C^\infty_0(\R^n)$, supported in $|x|\leq 1$ and satisfying $\phi(x)=1$ for
$|x|\leq \frac{1}{2}$. Thus we have
$$T_Nf(y)=T_Nf_1(y)+T_N^{(1)}f_2(y)+T_N^{(2)}f_2(y).$$
By taking $a=T_N^{(1)} f_2(x)$ we see that the mean value
$\frac{1}{|Q|}\int_Q |T_N f(y)-a|dy$ is bounded by
$$\frac{1}{|Q|}\left(\int_Q |T_N f_1(y)|dy+\int_Q |T_N^{(1)}f_2(y)- T_N^{(1)} f_2(x)| dy+\int_Q |T_N ^{(2)} f_2(y)|dy\right).$$
The first term is easy to handle. Indeed, we have $T_N= M
e^{-t_{N+1}H}$, and hence the boundedness of $M$ on
$L^2(\R^n)$ implies  that $T_N$ are uniformly bounded on $L^2(\R^n)$.
Consequently, \Bea \frac{1}{|Q|}\int_Q |T_N f_1(y)|dy&\leq&
\left(\frac{1}{|Q|}\int_Q |T_N f_1(y)|^2dy\right)^{\frac{1}{2}}\\
&\leq& C~ \left(\frac{1}{|Q|}\int_{2Q} |
f(y)|^2dy\right)^{\frac{1}{2}}\leq C~ \Lambda_2 f(x).\Eea This
takes care of the first term.\\

In order to deal with the second and third terms we make use of
the estimates on the kernels $M_j(x,y)$. Consider
$$\frac{1}{|Q|} \int_Q |T_N^{(1)}f_2 (y)-T_N^{(1)}f_2(x)| dy$$
$$\leq \frac{1}{|Q|} \int_Q \left(\int_{\R^n\setminus 2Q}|K_N^{(1)}(y,z)-K_N^{(1)}(x,z)||f(z)|dz\right) dy.$$
Let $u$ be the center of $Q$ and $l(Q)$ the side length of $Q$.
Since $ \R^n\setminus 2Q $ is contained in the union of the annuli  $ 2^kl(Q)<|z-u|< 2^{k+1} l(Q), k=1,2,\ldots  ,$  the inner integral above dominated by
$$\sum_{k=1}^\infty \int_{2^k l(Q)<|z-u|\leq 2^{k+1}l(Q)}|K^{(1)}_N(y,z)-K_N^{(1)}(x,z)||f(z)|dz.$$
Applying Cauchy-Schwarz inequality each integral in the above sum  is bounded by the product of
$$\sum_{k=1}^\infty \left(\int_{2^kl(Q)<|z-u|\leq 2^{k+1} l(Q)}|x-z|^{-n-1}|f(z)|^2
dz\right)^{\frac{1}{2}}$$
and $$\left(\int_{2^kl(Q)<|z-u|\leq  2^{k+1}
l(Q)} |K_N^{(1)}(y,z)-K_N^{(1)}(x,z)|^2 |x-z|^{n+1}
dz\right)^{\frac{1}{2}}.$$ Since $x\in Q$, $|x-u|\leq l(Q)$
whereas $|z-u|> 2^k l(Q)$. This means that
$$ |x-z| \geq |z-u|-|x-u| \geq (2^k-1)l(Q) \geq 2^{k-1}l(Q) $$ and therefore,
$$\left(\int_{2^k l(Q)<|z-u|\leq 2^{k+1}l(Q)} |x-z|^{-n-1}|f(z)|^2dz\right)^{\frac{1}{2}}$$
$$\leq C~ \left((2^{k-1}l(Q))^{-n-1}\int_{|z-u|\leq 2^{k+1}l(Q)}|f(z)|^2dz\right)^{\frac{1}{2}}$$
$$\leq C~ 2^{-k/2} l(Q)^{-\frac{1}{2}} \Lambda_2 f(x).$$\\

If we can show that for all $y\in Q$
$$\left(\int_{2^kl(Q)<|z-u| \leq 2^{k+1} l(Q)}
|K_N^{(1)}(y,z)-K_N^{(1)}(x,z)|^2 |x-z|^{n+1}
dz\right)^{\frac{1}{2}}\leq C~ l(Q)^{\frac{1}{2}},$$ then
$$\frac{1}{|Q|} \int_Q |T_N^{(1)}f_2(y)-T_N^{(1)}f_2(x)|dy
\leq C~ \left(\sum_{k=1}^\infty 2^{-k/2}\right) \Lambda_2 f(x)\leq C
\Lambda_2 f(x)$$ as desired.
Since $K_N=\sum_{j=0}^N M_j$ it is enough to show that
$$\sum_{j=0}^N \left(\int_{2^k l(Q)< |z-u|\leq 2^{k+1} l(Q)} |x-z|^{n+1}
|M_j^{(1)}(x,z)-M_j^{(1)}(y, z)|^2 dz\right)^{\frac{1}{2}}\leq C
l(Q)^{\frac{1}{2}}
$$
uniformly in $x, y$ and $N$. To this end we will prove the
following estimate:
$$\left(\int_{2^k l(Q)< |z-u|\leq 2^{k+1} l(Q)} |x-z|^{n+1}
|M_j^{(1)}(x,z)-M_j^{(1)}(y, z)|^2 dz\right)^{\frac{1}{2}}$$
$$\leq C~ l(Q)^{\frac{1}{2}}\min \left\{\frac{l(Q)^{\frac{1}{2}}}{t_{j+1}^{\frac{1}{4}}},\frac{t_{j+1}^{\frac{1}{4}}}{l(Q)^{\frac{1}{2}}}\right\}.$$
It is not difficult to show that the follwing series converges and
$$\sum_{j=0}^\infty \min \left\{\frac{l(Q)^{\frac{1}{2}}}{t_{j+1}^{\frac{1}{4}}},\frac{t_{j+1}^{\frac{1}{4}}}{l(Q)^{\frac{1}{2}}}\right\}\leq C$$
with a constant $ C $ independent of $ l(Q).$  Therefore by summing up over $ j $ we get the required estimate.\\

Since $x,y\in Q$ and $|z-u|>2^k l(Q)$, $|x-z|$ and $|y-z|$ are
comparable. Hence
$$\left(\int_{2^k l(Q)< |z-u|\leq 2^{k+1} l(Q)} |x-z|^{n+1}
|M_j^{(1)}(x,z)-M_j^{(1)}(y, z)|^2 dz\right)^{\frac{1}{2}}$$
$$\leq 2 \left(\int_{2^k l(Q)< |z-u|\leq 2^{k+1} l(Q)} |y-z|^{n+1}
|M_j^{(1)}(y, z)|^2 dz\right)^{\frac{1}{2}}$$
$$\leq C~ \left(\int_{\R^n} |y-z|^{n+1}
|M_j^{(1)}(y, z)|^2 dz\right)^{\frac{1}{2}}.$$ By
Corollary 2.4, the above  is bounded by  $t_{j+1}^{\frac{1}{4}}.$
In other words
$$\left(\int_{2^k l(Q)< |z-u|\leq 2^{k+1} l(Q)} |x-z|^{n+1}
|M_j^{(1)}(x,z)-M_j^{(1)}(y, z)|^2 dz\right)^{\frac{1}{2}}$$
$$\leq C~ l(Q)^{\frac{1}{2}}(t_{j+1}^{\frac{1}{4}} l(Q)^{-\frac{1}{2}}).$$\\

On the other hand by mean value theorem we can get another
estimate. By the definition of the truncated kernels
$$M_j(x,z) \phi(x-z)- M_j(y,z) \phi(y-z)$$
$$= (M_j(x,z)-M_j(y,z)) \phi(x-z)+M_j(y,z)(\phi(x-z)-\phi(y-z)).$$
By mean value theorem the above is bounded by
$$|x-y| (|\nabla_x M_j(\tilde{x}, z)||\phi(x-z)|+|\nabla_x \phi ( \tilde{y}-z)||M_j(y,z)|$$
$$\leq |x-y| |\nabla_x M_j(\tilde{x},z)||\phi(x-z)|+c |x-y| |M_j(y,z)|)$$
where $\tilde{x}$ and $ \tilde{y} $ are points on the line joining $x$ and $y$. The
integral corresponding to the second term is dominated by
$$\left(|x-y|^2 \int_{\R^n} |x-z|^{n+1}|M_j(y,z)|^2 dz\right).$$
Since $|x-z|$ is comparable to $|y-z|$ for $x, y\in Q$ and
$|x-y|\leq 2 l(Q)$, in view of Corollary 2.4 the above gives the
estimate
$$C~ l(Q)^2 t_{j+1}^{\frac{1}{2}} \leq C~ l(Q) \frac{l(Q)}{t_{j+1}^{\frac{1}{2}}}.$$
It therefore suffices to prove the estimate
$$\left(\int_{2^k l(Q)< |z-u|\leq 2^{k+1} l(Q)} |x-z|^{n+1}|\nabla M_j(\tilde{x}, z)|^2 |\phi(x-z)|^2 dz\right)^{\frac{1}{2}}\leq C~ t_{j+1}^{-\frac{1}{4}}.$$
Since $\phi$ is supported on $|y|\leq 1$, the integral is taken
over $|x-z|\leq 1$. When $y\in Q$, $|y-z|\leq |y-x|+|x-z|\leq 2
l(Q)+|x-z|$ and as $|x-z|\geq \frac{1}{2} l(Q)$ we see that
$|y-z|\leq 5 |x-z|\leq 5.$ Similarly, $|\tilde{x}-z|\leq 5$ for
any point on the line joining $x$ and $y$. Also, note that $|x-z|$
is comparable to $|\tilde{x}-z|$. Hence the above integral is
bounded by
$$\left(\int_{|\tilde{x}-z|\leq 5} |\tilde{x}-z|^{n+1}|\nabla M_j(\tilde{x},z)|^2dz\right)^{\frac{1}{2}}\leq C~ t_{j+1}^{-\frac{1}{4}},$$
in view of Proposition 2.5. Thus we have proved
$$\left(\int_{2^k l(Q)< |z-u|\leq 2^{k+1} l(Q)}|x-z|^{n+1}|M_j^{(1)}(x,z)- M_j^{(1)}(y,z)|^2dz\right)^{\frac{1}{2}}$$
$$\leq C~ l(Q)^{\frac{1}{2}}\min\{\frac{l(Q)^{\frac{1}{2}}}{t_{j+1}^{\frac{1}{4}}}, \frac{t_{j+1}^{\frac{1}{4}}}{l(Q)^{\frac{1}{2}}}\}$$
which is the desired estimate.\\

Finally, coming to the third term $T_N^{(2)}f_2$ we consider
$$M_j^{(2)}f_2(y)=\int_{\R^n} M_j(y,z) (1-\phi(y-z)) f_2(z)dz.$$
With $s_j=t_{j}^{\frac{1}{2}}$, $M_j^{(2)}f_2(y)$  is bounded by $$
 \int_{|z-y|\geq \frac{1}{2}} (1+
s_{j+1}^{-1}|z-y|^2)^{\frac{n+1}{4}}|M_j(y,z)|(1+
s_{j+1}^{-1}|z-y|^2)^{-\frac{n+1}{4}}|f(z)|dz$$
$$ \leq C~ s_{j+1}^{-\frac{n+1}{4}}
\int_{\R^n}|z-y|^{\frac{n+1}{2}}|M_j(y,z)|(1+
s_{j+1}^{-1}|z-y|^2)^{-\frac{n+1}{4}}|f(z)|dz $$
 where we have
used the fact that $|z-y|\geq \frac{1}{2} $ and $s_{j+1}^{-1}\geq
1$. By Cauchy-Schwarz the above gives \Bea |M_j^{(2)}
f_2(y)|&\leq& C~ s_{j+1}^{-\frac{n+1}{4}}\left(\int_{\R^n}
|y-z|^{n+1}|M_j(y,z)|^2dz\right)^{\frac{1}{2}}\\
&& \left(\int_{\R^n}
(1+s_{j+1}^{-1}|y-z|^2)^{-\frac{n+1}{2}}|f(z)|^2
dz\right)^{\frac{1}{2}}.\Eea
Since $ (1+|z|^2)^{-\frac{n+1}{2}} $ is radial and integrable, the second term is bounded by
$s_{j+1}^{\frac{n}{4}}\Lambda_2 f(y)$ ( see Theorem 2, Chapter III in \cite{S}) whereas by Corollary 2.4 the
first term is bounded by
$t_{j+1}^{\frac{1}{4}}=s_{j+1}^{\frac{1}{2}}$. Consequently
$$|M_j^{(2)}f_2(y)|\leq C~ s_{j+1}^{\frac{1}{4}}\Lambda_2f(y).$$
Therefore, \Bea \frac{1}{|Q|}\int_Q|T_N^{(2)} f_2(y)|dy &\leq& ~   \left(  \sum_{j=0}^N \frac{1}{|Q|}\int_Q |M_j^{(2)} f_2(y)|dy \right)\\
&\leq& C~ \left( \sum_{j=0}^N s_{j+1}^{1/4} \right) \frac{1}{|Q|} \int_Q \Lambda_2 f(y) dy.\Eea Taking sup
over all $Q$ containing $x$ we get
$$\frac{1}{|Q|}\int_Q |T^{(2)}_N f_2(y)|dy\leq C~ \Lambda(\Lambda_2 f)(x).$$
This completes the proof.\\

We are now in a position to prove Theorem 1.8. We will make use of
the relation between the sharp maximal function $\Lambda^\sharp$
and the dyadic maximal function $\Lambda_d$. We require (see Lemma
7.10 in \cite{D})
\begin{lem} Let $w\in A_p(\R^n)$, $1\leq p_0\leq p<\infty$. Then
$$\int_{\R^n}|\Lambda_d f(x)|^p w(x)dx\leq C~\int_{\R^n}|\Lambda^\sharp f(x)|^p w(x)dx$$
whenever $\Lambda_d f\in L^{p_0}(\R^n, w)$.
\end{lem}
From the pointwise estimate $|T_Nf(x)|\leq \Lambda_d (T_N f)(x),$
valid for a.e. $x,$ we get
$$\int_{\R^n} |T_N f(x)|^p dx\leq C~ \int_{\R^n} (\Lambda_d(T_Nf)(x))^pdx.$$
As $T_N$ are uniformly bounded on $L^2(\R^n)$, $\Lambda_d(T_N
f)\in L^2(\R^n)$, and hence by the above lemma
$$\int_{\R^n} |T_N f(x)|^pdx\leq C~ \int_{\R^n}|\Lambda^\sharp (T_N f(x))|^p dx$$
for any $p>2$. In view of Theorem 3.1 using the boundedness of
$\Lambda$ and $\Lambda_2$ on $L^p(\R^n)$, $p>2$ we obtain
$$\int_{\R^n} |T_N f(x)|^p dx\leq C~ \int_{\R^n} |f(x)|^pdx$$
valid on $L^p\cap L^2(\R^n)$, $p>2$. By the uniform boundedness
principle, we see that the limiting operator, namely $M$, has a
bounded extension to $L^p(\R^n)$, $p>2$.\\

As for the weighted inequality, once again we are led to check if
$\Lambda_d(T_N f)\in L^{p_0}(\R^n, w)$ for some $1\leq p_0\leq
p$, $w\in A_{p/2}$ on a dense class of functions. Since
$A_{p/2}\subset A_p$ and $\Lambda_d$ is bounded on $L^p(\R^n, w)$,
$w\in A_p$, it is enough to check that for all $f\in C_0^\infty
(\R^n)$, $T_Nf\in L^p(\R^n, w)$ when $w\in A_{p/2}$.\\

Suppose $f\in C^\infty_0(\R^n)$ is supported in $|x|\leq R$. Then
$$\int_{|x|\leq 2R} |T_N f(x)|^p w(x)dx\leq \left(\int_{|x|\leq 2R}w(x)^{1+\epsilon} dx\right)^{\frac{1}{1+\epsilon}}
\left(\int_{|x|\leq 2R}|T_N
f(x)|^{p\frac{1+\epsilon}{\epsilon}}dx\right)^{\frac{\epsilon}{1+\epsilon}}$$
By reverse H\"{o}lder inequality, we can choose $\epsilon>0$ so
that the first factor is also finite. For $|x|>2R$,
$$ |T_Nf(x)|\leq \int_{|y|\leq R} |K_N(x,y)| |x-y|^{\frac{n+1}{2}} |f(y)| |x-y|^{-\frac{n+1}{2}}dy$$
which gives the estimate, since $|x-y|>\frac{1}{2}|x|$,
$$|T_N f(x)|\leq C~ |x|^{-\frac{n+1}{2}}\left(\int_{|y|\leq R}|f(y)|^2 dy\right)^{\frac{1}{2}}\left(\int_{\R^n}|K_N(x,y)|^2|x-y|^{n+1}dy\right)^{\frac{1}{2}}.$$
If we use the estimates in Corollary 2.4 we get
$$|T_N f(x)|\leq C~ ||f||_\infty |x|^{-\frac{n+1}{2}}(\sum_{j=0}^N t_{j+1}^{\frac{1}{4}})\leq c(f) |x|^{-\frac{n+1}{2}}.$$
Therefore,
$$\int_{|x|>2R} |T_N f(x)|^p w(x)dx \leq C(f) \sum_{k=1}^\infty \int_{2^kR \leq |x| \leq 2^{k+1}R} w(x) |x|^{-\frac{n+1}{2}p}dx$$
which is bounded by
$$C(f) \sum_{k=1}^\infty (2^k R)^{-\frac{n+1}{2}p}\left(\int_{|x|\leq 2^{k+1}R}w(x)dx\right)$$
$$\leq C(f)\sum_{k=1}^\infty (2^k R)^{-\frac{n+1}{2}p}(2^{k+1}R)^{n p/2}< \infty$$
since $w(B(0, 2^{k+1}R))\leq C(R) (2^{k+1}R)^{p/2}$ as $w\in
A_{p/2}$.\\

Thus the hypothesis of Lemma 3.2 is satisfied and so for $p>2$,
$w\in A_{p/2}$ we get
$$\int_{\R^n} |T_N f(x)|^p w(x)dx\leq C~\int_{\R^n}|f(x)|^p w(x)dx.$$
Again by uniform boundedness principle the limiting operator $M$
is bounded on $L^p(\R^n, wdx)$.\\

\end{proof}
\section[The boundedness of $M$ on $L^p(\R^n)$]{The boundedness of $M$ on
$L^p(\R^n)$} We first observe that Theorem 1.6 follows from
Theorem 1.8. Indeed, we have the boundedness of $M$ on $L^p(\R^n)$
for $p>2$ and as $M^*$ also satisfies the same hypothesis as $ M,$
by duality we obtain the boundedness on $1<p<2$ as well. We now
proceed to prove Theorem 1.7 and 1.9. We need the following
analogue of Proposition 2.3.
\begin{prop}
Assume that $M$ satisfies the hypothesis of Theorem 1.7. Then with
the same notation as in section 2, we have
$$\sup_{x,y\in \R^n}|x-y|^l |M_j(x,y)|\leq C~ t_{j+1}^{(l-n)/2}$$
for all $l\leq n+1$.
\end{prop}
\begin{proof}
As in the proof of Proposition 2.3, it is enough to estimate the
$L^{\infty}$-norm of \be
\sum_{N=0}^{\infty}\int_{\R^n}((\delta^\mu\bar{\delta}^\nu
M)\chi_N)(x,y')(\chi_N(\delta^\gamma \bar{\delta}^\rho
S_j))(y',y)dy'\ee for all $\mu, \nu, \gamma, \rho$ satisfying
$|\mu|+|\nu|+|\gamma|+|\rho|=l$. By Cauchy-Schwarz inequality
$$|\int_{\R^n}((\delta^{\mu}\bar{\delta}^{\nu}M)\chi_N)(x,y')(\chi_N(\delta^\gamma \bar{\delta}^\rho S_j)(y',y)dy'|$$
$$\leq \left(\int_{\R^n}|(\delta^\mu \bar{\delta}^\nu
M)\chi_N)(x,y')|^2dy'\right)^{\frac{1}{2}}
\left(\int_{\R^n}|(\chi_N(\delta^{\gamma}\bar{\delta}^\rho
S_j))(y',y)dy'|^2\right)^{\frac{1}{2}}$$
$$\leq C~2^{\frac{N}{2}(\frac{n}{2}-|\mu|-|\nu|)} (t_{j+1} 2^N) 2^{\frac{N}{2}(\frac{n}{2}-|\gamma|-|\rho|)}f_{\gamma, \rho}(2^Nt_{j+1})^{\frac{1}{2}}$$
where we have used the estimate 2.2. Hence $L^\infty$-norm of
(4.4) is bounded by
$$ t_{j+1} ~ \sum_{N=0}^\infty 2^{\frac{N}{2}(n-l+2)} f_{\gamma, \rho}(2^N t_{j+1})^{\frac{1}{2}}\leq C~ t_{j+1}^{(l-n)/2}$$
since the functions $ f_{\gamma,\rho} $ have exponential decay.
\end{proof}
\begin{cor}
Under the hypothesis of Theorem 1.7 we have the estimate
$$\sup_{x,y}|x-y|^{n+\frac{1}{2}}|M_j(x,y)|\leq C~
t_{j+1}^{\frac{1}{4}}.$$

\end{cor}
In the next proposition we will prove $L^\infty$ estimates of the
gradient of the kernel of $M_j$.
\begin{prop}
Under the hypothesis of Theorem 1.7 we have the estimates
$$\sup_{|x-y|\leq K} |x-y|^l |\nabla_x M_j(x,y)||\leq C_K
t_{j+1}^{(l-n-1)/2}$$ and
$$\sup_{x,y\in \R^n} |x-y|^l |\nabla_y M_j(x,y)||\leq C
t_{j+1}^{(l-n-1)/2}$$ whenever $l\leq n+1$.\\

\end{prop}
\begin{proof} We have already seen that in order to estimate the
kernel of $\frac{\partial}{\partial x_j}M_j(x,y)$, it is enough to
consider the kernels of $A_iM_j$ and $A^*_i M_j$. We will consider
only $A_iM_j$, the other case being similar. As $A_iM_j=\delta_i
M_j+M_j A_i$, we consider $\delta_i M_j$ and $M_jA_i$ separately.
Again we have $$\delta_iM_j=(\delta_iM)S_j+M(\delta_i S_j)$$ and let us first
consider $(\delta_iM)S_j$. We have already seen that for
estimating $|x-y|^l|((\delta_i M)S_j)(x,y)|$, it is enough to
consider
$$\sum_{N=0}^\infty ((\delta^\mu \bar{\delta}^\nu(\delta_i M))\chi_N.\chi_N(\delta^\gamma \bar{\delta}^\rho S_j))(x,y)$$
where $|\mu|+|\nu|+|\gamma|+|\rho|=l$. But for $l=n+1$, we do not
have any estimate for $(\delta^\mu \bar{\delta}^\nu (\delta_i
M))\chi_N$, when $|\mu|+|\nu|=n+1$ as $M$ is only of class $C^{n+1}$.

So, in order to to estimate $|x-y|^l|(\delta_iM)S_j(x,y)|$ we first estimate
$|x-y|^{l-1}|(\delta_iM)S_j(x,y)|$ for $l\leq n+1$ and then use the
fact $|x-y|\leq K$ to get the required estimate. But
$|x-y|^{l-1}|(\delta_i M_j(x,y)|$ can be estimated as in the proof of
Proposition 2.5 which gives us the following
$$\sup_{x,y\in\R^n} |x-y|^{l-1}|(\delta_i M_j)(x,y)|\leq C~ t_{j+1}^{(l-n)/2}.$$ Once we have this
$$\sup_{|x-y|\leq K}|x-y|^l|((\delta_iM)S_j)(x,y)|
\leq K \sup_{x,y\in \R^n}|x-y|^{l-1}|((\delta_i M)S_j)(x,y)|$$
gives the better estimate $ C~ t_{j+1}^{(l-n)/2}.$
Hence we get the required estimate for $(\delta_i M)S_j$.

In a similar way $|x-y|^{l-1}M(\delta_i S_j)$ can be estimated to
get the bound
$$\sup_{x,y\in \R^n}|x-y|^{l-1}|M(\delta_iS_j)(x,y)|\leq C~t_{j+1}\sum_{N=0}^\infty 2^{N(n+2-l)/2} g(2^N t_{j+1})\leq C~t_{j+1}^{(l-n)/2}$$
where $g$ is a rapidly  decreasing function. Thus by a similar argument used  for the case of $(\delta_i M)S_j,$ we
obtain the required estimate. Note that, in this case also we cannot estimate $ |x-y|^l(M(\delta_i S_j))(x,y)$ directly for $l=n+1$
as in the estimation we will then have $C~ t_{j+1}\sum_{N=0}^\infty
2^{N(n+1-l)}g(2^Nt_{j+1})$ which is not good enough.

In order to estimate $M_j A_i$ we need to establish the following
estimate: \be\sup_{x\in \R^n} \int_{\R^n} |(\chi_N(\delta^\gamma
\bar{\delta}^\rho(S_j A_i))(x,y)|^2dy\ee
$$\leq C~ t_{j+1}^2
2^{N(\frac{n}{2}+3-|\gamma|-|\rho|)}f_{\gamma,\rho}(2^N
t_{j+1}).$$ Here, we can calculate $|x-y|^l (M_j A_i)(x,y)$
directly as there is no extra non-commutative derivative falling
on $M_j$. Once again, in order to estimate the above term , it is
enough to estimate the $L^\infty$ norm of
$$\sum_{N=0}^\infty \int_{\R^n}|((\delta^{\mu}\bar{\delta}^\nu M)\chi_N)(x, y'). (\chi_N (\delta^{\gamma}\bar{\delta}^\rho(S_j A_i)))(y',y)dy'.$$
Applying Cauchy-Schwarz inequality the above term is bounded by
\Bea&&\sum_{N=0}^\infty \left(\int_{\R^n}|((\delta^{\mu} \bar{\delta}^\nu
M)\chi_N)(x,y')|^2 dy'\right)^{\frac{1}{2}}
\left(\int_{\R^n}|\chi_N(\delta^{\gamma}\bar{\delta}^\rho(S_j A_i)(y',
y)|^2 dy'\right)^{\frac{1}{2}}\\ &\leq& C~ t_{j+1}\sum_{N=0}^\infty
2^{N(\frac{n}{2}-|\mu|-|\nu|)/2}2^{N(\frac{n}{2}+3-|\gamma|-|\rho|)/2}f_{\gamma,
\rho}(2^N t_{j+1})\\ &\leq& C~ t_{j+1}\sum_{N=0}^\infty
2^{N(n+3-l)/2}f_{\gamma, \rho}(2^N t_{j+1}) \\ &\leq& C~
t_{j+1}^{(l-n-1)/2}.\Eea Note that in the above estimation the case $l=n+1$
 does not create any problem. Hence, we have proved the first
part of the proposition.

For the second part, in order to estimate
$|x-y|^l|\frac{\partial}{\partial y_i}M_j(x,y)|$ it is enough to
estimate $|x-y|^l|M_jA_i(x,y)|$ and $|x-y|^l|M_j A_i^*(x,y)|$ since
for any $f\in C^\infty_c(\R^n)$ it is easy to verify that \Bea
\int_{\R^n} \frac{\partial}{\partial y_i} M_j(x,y) f(y)dy &=&
-\int_{\R^n}M_j(x,y)\frac{\partial f}{\partial y_i}(y)dy\\
&=& \frac{1}{2} (M_j(A_i- A^*_i))f(x).\Eea But then we have already
estimated $|x-y|^l|M_j A_i(x,y)|$ in the first part. Note that, as
we have directly estimated $|x-y|^l|M_j A_i(x,y)|$, we did not use
the assumption $|x-y| \leq A$ for estimating $M_jA_i(x,y)$. This
observation allows us to take supremum over all $x$ and $y$ in
$\R^n$ for the second part.

 \end{proof}
 \begin{cor}Under the hypothesis of Theorem 1.7 we have the estimates
$$\sup_{|x-y|\leq A} |x-y|^{n+\frac{1}{2}} |\nabla_x M_j(x,y)||\leq C
t_{j+1}^{-1/4}$$ and
$$\sup_{x,y\in \R^n} |x-y|^{n+\frac{1}{2}} |\nabla_y M_j(x,y)||\leq C
t_{j+1}^{-1/4}.$$
\end{cor}

We are now in a position to prove Theorem 1.7. By Theorem 1.8 we
 have already proved that M is bounded on $L^p$ for $2\leq
 p<\infty$. Now we will show that $T_N$ satisfies weak (1,1)
 estimate with a uniform constant. We shall show that the kernel $M_j$ satisfies the following estimate:
$$ |x-z|^{n+\frac{1}{2}} |M_j(x,y)-M_j(x,z)|
\leq C~|y-z|^{1/2}
\min\left\{\frac{t_{j+1}^{1/4}}{|y-z|^{1/2}},
\frac{|y-z|^{1/2}}{t_{j+1}^{1/4}}\right\}$$
 whenever $|x-z|>2|y-z|$ and
$|x-y|> 2|y-z|$. Once we have the above estimate, we can prove that
$$|x-z|^{n+\frac{1}{2}}|K_N(x,y)-K_N(x,z)|\leq C~|y-z|^{1/2}$$
where $C$ is independent of $N$. Hence, $T_N$ is of weak type (1,1) (See
Theorem 5.10 of \cite{D}) with a uniform constant.\\

From Corollary 4.2 we get the estimate
$$|x-z|^{n+\frac{1}{2}}|M_j(x,y)-M_j(x,z)|\leq C |y-z|^{1/2}\frac{t_{j+1}^{1/4}}{|y-z|^{1/2}}.$$
Again , using mean value theorem
$$|x-z|^{n+\frac{1}{2}}|M_j(x,y)-M_j(x,z)|$$
$$\leq |x-z|^{n+\frac{1}{2}}|y-z||\nabla_y M_j(x, \tilde{z})|$$
where $\tilde{z}$ is a point on the line joining $z$ and $y$. Since
$|x-z|$ and $|x- \tilde{z}|$ are comparable, using Corollary 4.4 the
above can be dominated by $C~t_{j+1}^{-1/4}|y-z|$. Combining the above two we get the required estimate.\\

In order to prove Theorem 1.9 we need the following analogue of
Theorem 3.1:
\begin{thm}
Let $M$ satisfies the hypothesis of Theorem 1.9. Then we have
$$\Lambda^{\sharp}(\sum_{j=1}^N M_j f) \leq C~ (\Lambda_s f(x)+ \Lambda (\Lambda f)(x))$$
for $f\in L^p(\R^n)$ where $s<p$, $1<s<\infty$, $N\in \N$ and $C$
is independent of $N.$
\end{thm}
The proof of the theorem is similar to that  of Theorem 3.1. Once Theorem
4.5 is proved it is  routine to prove Theorem 1.9. We leave the details to the reader.\\

\section[Hermite pseudo-multipliers]{Hermite pseudo-multipliers} In
this section we prove the results concerning Hermite
pseudo-multipliers. Given a bounded function $a$ on $\R^n\times
\R$ which is $2\pi$-periodic in the second variable consider
$$\hat{a}(x, H)f(x)=\sum_{k=0}^\infty \hat{a}(x, 2k+n)P_k f(x)$$
where $\hat{a}(x,k)=\int_0^{2\pi} a(x,t) e^{-i kt}dt$. If we let
$$e^{-it H}f=\sum_{k=0}^\infty e^{-i t(2k+n)}P_k f$$
stand for the unitary group generated by $H$, we have
$$\hat{a}(x,H)f(x)=\int_0^{2\pi} a(x,t) e^{-i t H}f(x)dt.$$
Consequently $\hat{a}(x,H)$ is bounded on $L^2(\R^n)$ provided
$a(x,t)\in L^\infty (\R^n\times \R)$. This proves the $L^2$ case
of Corollary 1.5.\\

Actually, as indicated in the introduction with $
Da(x,t)=\frac{\partial}{\partial t}((e^{-it}-1)a(x,t))$, it
follows that
$$i k \Delta \hat{a}(x,k)=\int_0^{2\pi} Da(x, t) e^{-i tk}dt.$$
Consequently, the hypothesis of Theorem 1.3 on $\hat{a}(x,k)$ will
be satisfied provided, $D^ja$ and $D^j \frac{\partial}{\partial
x_i}a$ belong to $L^\infty(\R^n\times \R)$ for all $j=0, 1, 2,
\cdots, [\frac{n}{2}]+1, i = 1,2, \cdots, n$. Hence Corollary 1.5
follows immediately
once we prove Theorem 1.3.\\

We prove Theorem 1.3 by showing that under the hypothesis on
$m(x,k)$, the operator $m(x, H)$ satisfies the kernel estimates
used in the proof of Theorem 1.8. A close examination of the proof of
Theorem 1.8 shows that the following variant of the theorem is
true.\\

Let us introduce the notation $D_jM=[x_j, M]$ for the commutator
of $M$ with the operator of multiplication by $x_j$. For a
multi-index $\alpha$, we let $D^\alpha M=D_1^{\alpha_1}\cdots
D_n^{\alpha_n}M.$ Note that $2D_j=\delta_j+\bar{\delta_j}$\\

\begin{thm}
Let $M$ be a bounded linear operator on $L^2(\R^n)$ such that
$D^\alpha M$, $D^\beta \delta_j M$ are all bounded on $L^2(\R^n)$
for $j=1, 2, \cdots, n$ and $|\alpha|\leq [\frac{n}{2}]+1$,
$|\beta|\leq [\frac{n}{2}]$. Assume that
$$\sup_{x\in \R^n} \int_{\R^n}|((D^\alpha M)\chi_N)(x,y)|^2 dy\leq C~ 2^{N(\frac{n}{2}-|\alpha|)},$$
$$\sup_{x\in \R^n}\int_{\R^n}|((D^\beta \delta_j M)\chi_N) (x,y)|^2dy\leq C~ 2^{N(\frac{n}{2}-|\beta|)}$$
for $|\alpha|\leq [\frac{n}{2}]+1$ and $|\beta|\leq
[\frac{n}{2}]$. Then for any $p>2$, $w\in A_{p/2}$ we have
$$\int_{\R^n}|Mf(x)|^p w(x) dx\leq C~ \int_{\R^n} |f(x)|^p w(x)dx.$$
\end{thm}

Observe that $D_jM=[x_j, M]$ are also derivations and hence
$D^\alpha(M S_j)$ will be a linear combination of $(D^\beta
M)(D^\gamma S_j)$, $|\beta|+|\gamma|=|\alpha|$. In proving
Proposition 2.3 we have started with the observation that
$(x-y)^\alpha M_j(x,y)$ is the kernel of $D^\alpha M_j$. The
derivatives $D^\gamma S_j$ can be written in terms of the
derivatives $\delta^\mu$ and $\bar{\delta}^\nu$ and hence we have
an analogue of Proposition 2.1 for $D^\alpha S_j$. The hypothesis
on $D^\alpha M$ can be used instead of the hypothesis on $
\delta^\alpha \bar{\delta}^\beta M$. Proposition 2.5 is similarly
proved
using the assumption on $D^\beta \delta_j M.$\\

Once we have the kernel estimates stated in Proposition 2.3 and
2.5 we can proceed as before to prove Theorem 5.1.\\

When $M=m(x,H)$ its kernel is given by
$$M(x,y)=\sum_{k=0}^\infty m(x, 2k+n)\Phi_k(x,y)$$
where $\Phi_k(x,y)=\sum_{|\alpha|=k}\Phi_\alpha(x)\Phi_\alpha (y)$
are the kernels of $P_k$. The kernel of $D^\alpha M$ is then
$(x-y)^\alpha M(x,y)$ and so the kernel of $(D^\alpha M)\chi_N$ is
given by
$$((D^\alpha M)\chi_N)(x,y)=\int_{\R^n} (x-y')^{\alpha}M(x,y')\left(\sum_{2^{N-1}\leq 2k+n< 2^N} \Phi_k(y', y)\right) dy'.$$
Since $\int_{\R^n} \Phi_{k}(x,y')\Phi_j(y',y)dy'=\delta_{kj}
\Phi_k(x,y)$ we have
$$(M\chi_N)(x,y)=\sum_{2^{N-1}\leq 2k+n< 2^N} m(x, 2k+n) \Phi_k(x,y).$$
which gives the estimate
$$\int_{\R^n}|(M\chi_N)(x,y)|^2dy\leq C~ \sum_{2^{N-1}\leq 2k+n< 2^N} (m(x, 2k+n))^2 \Phi_k(x,x). $$
If we make use of the estimate $\Phi_k(x,x)\leq C~
k^{\frac{n}{2}-1}$ proved in \cite{T} (see Lemma 3.2.2) we obtain
$$\int_{\R^n}|(M\chi_N)(x,y)|^2dy\leq C~ 2^{N\frac{n}{2}}$$
since $m$ is assumed to be a bounded function of $x$ and $k$.\\

In order to estimate the $L^2$ norm of the kernel of $(D^\alpha
M)\chi_N$ we need to get an expression for its kernel. Let us
write $A^*_j=-\frac{\partial}{\partial x_j}+x_j$ and
$B_j^*=-\frac{\partial}{\partial y_j}+y_j$ each acting in the $x$
and $y$ variables respectively.
\begin{lem}
$$(x-y)^\alpha M(x,y)=\sum C_{\beta, \gamma} \left ( \sum_{k=0}^\infty \Delta^{|\gamma|}m(x, 2k+n)(B^*-A^*)^\beta \Phi_k(x,y)\right )$$
where the sum is extended over all $\beta, \gamma$ satisfying
$2\gamma_j-\beta_j= \alpha_j$, $\gamma_j\leq \alpha_j$.
\end{lem}
This lemma has been proved in \cite{T} when $m(x,k)$ is assumed to be 
independent of $x$ (see Lemma 3.2.3 in \cite{T}). The same proof
can be modified to prove the above lemma.
\begin{prop}
Let $M=m(x,H)$ where $m(x, k)$ satisfies the hypothesis of Theorem
1.3. Then the kernels of $D^\alpha M$ and $D^\alpha \delta_j M$
satisfy the estimates stated in Theorem 5.1. Consequently, $M$
extends to a bounded operator on $L^p(\R^n, wdx)$, $w\in A_{p/2},
p>2$.
\end{prop}
\begin{proof}
For each $\beta$ and $\gamma$ satisfying
$2\gamma_j-\beta_j=\alpha_j$, $\gamma_j\leq \alpha_j$, we want to
estimate the $L^2$ norm (in y-variable) of the kernel \be
\sum_{k=0}^\infty \Delta^{|\gamma|}m(x,
2k+n)\int_{\R^n}(B^*-A^*)^\beta \Phi_k(x,y')\chi_N(y',y)dy'.\ee
Expanding $(B^*-A^*)^\beta$ and recalling the definitions of
$\Phi_k$ and $\chi_N$ we need to consider
$$\sum_{|\mu|=k}\sum_{2^{N-1}\leq 2|\nu|+n < 2^N} A^{*r} \Phi_\mu (x) \Phi_\nu(y)\int_{\R^n}B^{*s} \Phi_\mu(y') \Phi_\nu(y')dy'$$
where $r,s$ are multiindices such that $r+s=\beta$. Since $ B^*_j
\Phi_\mu(y)=(2\mu_j+2)^{\frac{1}{2}}\Phi_{\mu+e_j}$,
$e_j=(0,\cdots, 1,\cdots, 0)$
$$\int_{\R^n}B^{*s}\Phi_{\mu}(y')\Phi_\mu(y')dy'=C(\mu, s)\int_{\R^n}\Phi_{\mu+s}(y')\Phi_{\nu}(y')dy'$$
where $|c(\mu,s)|\leq C |\mu|^{\frac{1}{2}|s|}$. Due to the
orthogonality of Hermite functions the above sum reduces to
$$\sum_{2^{N-1}\leq k+|s|< 2^N}|\Delta^{|\gamma|} m(2k+n)|^2\sum_{|\mu|=k}C(\mu, r)C(\mu,s) \Phi_{\mu+r}(x) \Phi_{\mu+s}(y).$$
The square of the $L^2$-norm of this sum is
$$\sum_{2^{N-1}\leq k+|s|<2^N} |\Delta^{|\gamma|}m(x, 2k+n)|^2 \sum_{|\mu|=k} C(\mu, s)^2 C(\mu, s)^2 \Phi_{\mu+r}(x)^2.$$
In view of the hypothesis on $m$, the above is bounded by
$$\sum_{2^{N-1}\leq k+|s|<2^N}(2k+n)^{-2|\gamma|+|r|+|s|} (k+|r|)^{\frac{n}{2}-1}\leq C~ 2^{N(\frac{n}{2}-|\alpha|)}$$
since $|r|+|s|=\beta$, $2|\gamma|-|\beta|=|\alpha|$. Estimating
the kernels of $(D^\beta \delta_j M)$ is similar. This completes
the proof of Proposition 5.3 and hence Theorem 1.3 in proved.
\end{proof}

In order to prove Theorem 1.2  we need the following analogue
of Theorem of 5.1.
\begin{thm}
Let $M$ be a bounded linear operator on $L^2(\R^n)$ such that
$D^\alpha M$ are all bounded on $L^2(\R^n)$ for $|\alpha|\leq
n+1$. Assume that
$$\sup_{x\in \R^n}\int_{\R^n}|(D^\alpha M)\chi_N (x,y)|^2 dy\leq C 2^{N(\frac{n}{2}-|\alpha|)}$$
for $|\alpha|\leq n+1$. Then $M$ is of weak type $(1,1)$
and strong type $(p,p)$ for $1<p<2$.
\end{thm}
To prove the above result, it is enough to show that the estimates in Proposition
4.1 and in the second part of Proposition 4.3 will also be true under the
hypothesis of the above theorem. We have already discussed how one
can estimate $|x-y|^l M_j(x,y)$ in Theorem 5.1. From the proof of
Proposition 4.3 observe that for estimating $|x-y|^l|\nabla_y
M_j(x,y)|$ we basically need to estimate $D^\beta M$ and
$D^\gamma(S_j A_i)$, for $|\beta|+|\gamma|=l$ and $i=1, \cdots ,
n$. $D^\beta M$ can be estimated by the given hypothesis whereas
for $D^\gamma(S_j A_i)$ we need to use (4.5).

Once Theorem 5.4 is proved, the rest of the proof of Theorem 1.2
is similar to that of  Theorem 1.3.\\

\begin{center}
\bf{Acknowledgments}
\end{center}The first author is thankful to CSIR, India, for
the financial support. The work of the second author is supported
by J. C. Bose Fellowship from the Department of Science and
Technology (DST) and also by a grant from UGC via DSA-SAP. We would like to thank Rahul Garg for his careful reading of an earlier version of this article and making several useful suggestions.

\end{document}